\documentclass{article}
\usepackage{epsf}
\usepackage{latexsym}
\usepackage{amsfonts}

\newtheorem{theorem}{Theorem}[section]
\newtheorem{corollary}[theorem]{Corollary}

\newtheorem{definition}[theorem]{Definition}
\newenvironment{proof}{{\sc Proof:}}{~\hfill $\quad\Box$}

\begin{document}

\title{Taxicab Triangle Incircles and Circumcircles}

\author{Kevin P. Thompson}

\date{}

\thispagestyle{empty}
\renewcommand\thispagestyle[1]{} 

\maketitle
\begin{abstract}
Inscribed angles are investigated in taxicab geometry with application to the existence and uniqueness of
inscribed and circumscribed taxicab circles of triangles.
\end{abstract}

\section{Introduction}

In Euclidean geometry, all triangles have a unique incircle (inscribed
circle) and a unique circumcircle (cirumscribed circle). In taxicab
geometry, the shape of a circle changes to a rotated square \cite{Krause}. Therefore,
it is a logical question whether the Euclidean incircle and circumcircle
theorems for triangles still hold. We will see that only under certain
conditions do these circles exist in the traditional sense.

When describing angle sizes we will use the definition of a t-radian \cite{ThompsonDray}
which is the natural analogue in taxicab geometry of the radian in Euclidean geometry.

\begin{figure}[b]
\centerline{\epsffile{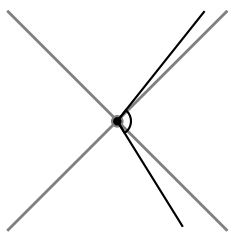}}
\caption{A taxicab angle that is not an inscribed angle}
\label{uninscribedangle}
\end{figure}

\section{Inscribed Angles}

In Euclidean geometry, all angles less than $\pi$ radians can be represented as an inscribed angle. This is not the
case in taxicab geometry (see Figure \ref{uninscribedangle}). This simple fact has far reaching consequences regarding
inscribed and circumscribed circles for triangles. And, since the position of an angle plays a role in
whether the angle is inscribed, we will begin with a number of definitions for different types of
inscribed angles in taxicab geometry.

\begin{definition}
A taxicab angle is {\em positively inscribed} if a line of slope 1 through its vertex remains outside the angle.
\end{definition}

\begin{definition}
A taxicab angle is {\em negatively inscribed} if a line of slope -1 through its vertex remains outside the angle.
\end{definition}

\begin{definition}
A taxicab angle is {\em inscribed} if it is positively inscribed, negatively inscribed, or both.
\end{definition}

\begin{definition}
A taxicab angle is {\em completely inscribed} if it is both positively inscribed and negatively inscribed.
\end{definition}

\begin{definition}
A taxicab angle is {\em strictly positively inscribed} if it is positively inscribed but not negatively inscribed.
Similarly, a taxicab angle is {\em strictly negatively inscribed} if it is negatively inscribed but not positively
inscribed.
\end{definition}

\begin{figure}[t]
\centerline{\epsffile{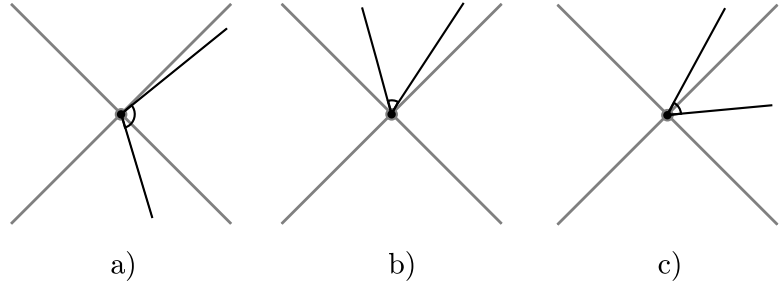}}
\caption{Inscribed taxicab angles. The angles shown are a) strictly positively inscribed, b) completely inscribed, and c) strictly negatively inscribed}
\label{inscribedangles}
\end{figure}

From the examples in Figure \ref{inscribedangles} it is clear that both the size of an
angle and its position affect whether it is inscribed. Some angles just
larger than 2 t-radians are not inscribed while other angles just
smaller than 4 t-radians are inscribed.

While not the focus of this investigation, it should be noted that the inscribed angle theorem from Euclidean
geometry does not hold in taxicab geometry. Figure \ref{inscribedangletheorem} illustrates two scenarios where
the ratio of the inscribed angle to the central angle subtended by the same arc is different. In the left figure,
$\alpha$ has measure 1 t-radian and $\theta$ has measure 2 t-radians. In the right figure, $\alpha$ again has
measure 1 t-radian but $\theta$ has measure 2.5 t-radians.

\begin{figure}[t]
\centerline{\epsffile{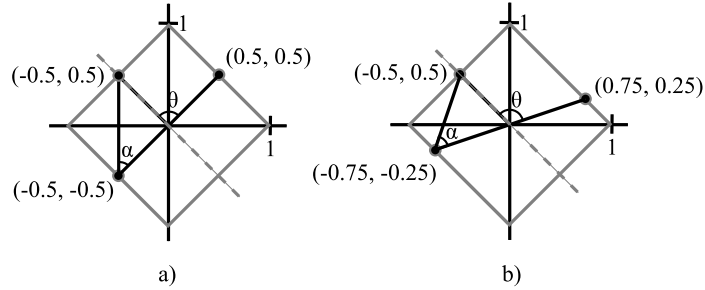}}
\caption{Failure of the Euclidean inscribed angle theorem in taxicab geometry}
\label{inscribedangletheorem}
\end{figure}

Triangles with one or more inscribed angles have special properties. A number
of these will be needed in our investigation of taxicab triangle circumcircles
and incircles.

\begin{theorem}
\label{neighboringanglestheorem}
If a triangle contains a strictly positively inscribed angle, then the other angles are negatively
inscribed.
\end{theorem}
\begin{proof} If an angle $\alpha$ of a triangle is strictly
positively inscribed, then for a neighboring angle to not be negatively
inscribed it would have to be larger than $(4-\alpha)$ t-radians
(see Figure \ref{positiveneighbor}) thus violating the angle requirements of a
triangle in taxicab geometry (Theorem 14 of \cite{ThompsonDray}).
\end{proof}

\begin{corollary}
\label{neighboringanglescorollary}
If a triangle contains a strictly negatively
inscribed angle, then the other angles are positively inscribed.
\end{corollary}

A triangle with a completely inscribed angle can only take on certain shapes and positions. Adding additional
completely inscribed angles results in a very small class of triangles as the following theorem ilustrates.

\begin{figure}[b]
\centerline{\epsffile{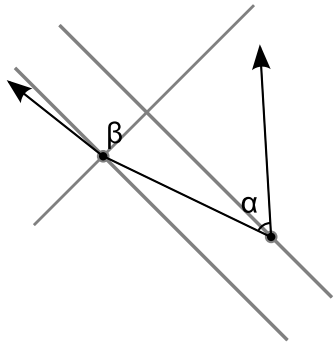}}
\caption{The neighboring angles of a strictly positively inscribed angle of a triangle must be negatively inscribed}
\label{positiveneighbor}
\end{figure}

\begin{theorem}
\label{threeinscribedanglestheorem}
A triangle has three completely inscribed angles if and only if two of its sides have slope -1 or slope 1.
\end{theorem}
\begin{proof}
Suppose $\alpha$ and $\beta$ are completely inscribed angles at vertices $A$ and $B$ of a triangle. 
Without loss of generality, suppose $\alpha$ is inscribed such that $AB$ and $AC$ have slope between -1 and 1
with $AB$ having a slope strictly less than 1 (Figure \ref{threeinscribedangles}). Since $\beta$
is also completely inscribed, the slope of $BC$ is between -1 and 1, inclusive. If the slope of $BC$ is strictly less than 1, then there
exists a line of slope 1 through C passing through the angle $\gamma$ making it not positively inscribed. Similarly, if the slope of $AC$
is strictly greater than -1, then there exists a line of slope -1 through C passing through the angle $\gamma$ making it not negatively
inscribed. Therefore, the triangle has three completely inscribed angles if and only if the rays of angle $\gamma$
have slope 1 and -1.
\end{proof}

\begin{figure}[t]
\centerline{\epsffile{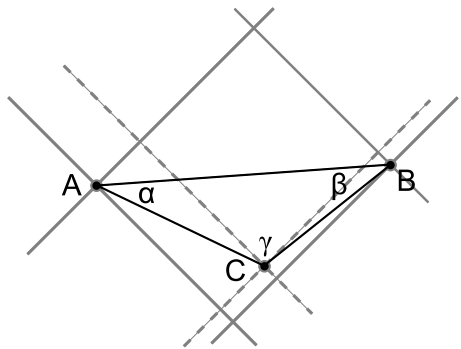}}
\caption{Defining a triangle with three completely inscribed angles}
\label{threeinscribedangles}
\end{figure}

\section{Inscribed Triangles}

Triangles whose angles are all inscribed will be of great interest to us. Having a term for this type of triangle will be
useful and, in the end, meaningful. While it may seem presumptuous, calling such a triangle an inscribed triangle will be
fully justified by the results that follow.

\begin{definition}
A triangle is {\em inscribed} if all of its angles are inscribed.
\end{definition}

Exploring some properties of inscribed triangles will aid in our study of taxicab triangle circumcircles and incircles.
There are minimal properties that an inscribed triangle must have. In fact, by imposing the restriction that all three angles must be
inscribed, a more stringent condition results. The following theorem illustrates one of the more fundamental
and useful properties of an inscribed triangle.

\begin{theorem}
\label{onecompletelyinscribedtheorem}
An inscribed triangle has at least one completely inscribed angle.
\end{theorem}
\begin{proof}
By way of contradiction and without loss of generality, assume a triangle
has a strictly positively inscribed angle $\alpha$. By
Theorem \ref{neighboringanglestheorem} the neighboring angles must be negatively
inscribed. Suppose one of these angles $\beta$ is strictly negatively
inscribed. Then by Corollary \ref{neighboringanglescorollary}, the other angles must be positively inscribed. So, $\gamma$ is completely inscribed
(Figure \ref{onecompletelyinscribed}). Therefore, an inscribed triangle cannot be constructed with no completely inscribed angles.
\end{proof}

\begin{figure}[t]
\centerline{\epsffile{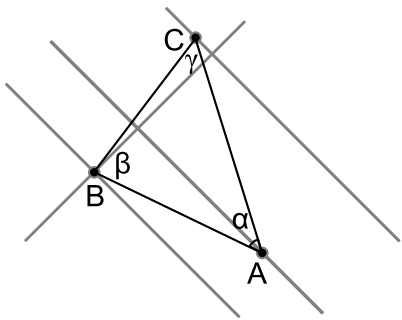}}
\caption{An inscribed triangle must have at least one completely inscribed angle}
\label{onecompletelyinscribed}
\end{figure}

The following theorem is an extension and generalization of Theorem \ref{neighboringanglestheorem} where a hint of alternating
inscribed angle properties among the angles of a triangle was seen. Theorem \ref{onecompletelyinscribedtheorem} above provides the flexibility in obtaining
the following result. The alternating property in an inscribed triangle is key to the construction of taxicab circumcircles for triangles.

\begin{theorem}
\label{neighboroppositetheorem}
Neighboring angles of an inscribed triangle have opposite inscribed angle properties (completely inscribed angles
may be selected to be either positively or negatively inscribed as needed).
\end{theorem}
\begin{proof}
If an angle of a triangle is strictly positively inscribed, then by Theorem \ref{neighboringanglestheorem}
the other angles are negatively inscribed. By Theorem \ref{onecompletelyinscribedtheorem} one of these angles must be
completely inscribed. Therefore, neighboring angles in this triangle have opposite
inscribed angle properties. The result is similar when beginning with a strictly negatively inscribed angle.

The only remaining case is a triangle with only completely inscribed angles. The result follows immediately.
\end{proof}


\section{Taxicab Triangle Circumcircles}

The traditional definition of a triangle circumcircle, or circumscribed
circle, is a circle that passes through the vertex points of the triangle.
In Euclidean geometry, all triangles have a unique circumcircle. In
taxicab geometry, the existence of circumcircles is more restrictive. 

\begin{theorem}
\emph{\textbf{(Triangle Circumcircle Theorem)}} A triangle has a
taxicab circumcircle if and only if it is an inscribed triangle.
\end{theorem}
\begin{proof}
If a taxicab circumcircle passes through the vertices of a triangle and completely
encloses the triangle, by definition the angles of the triangle are inscribed since a taxicab circle
is composed of lines of slope 1 and -1.

To prove the converse, suppose the triangle is an inscribed triangle. By Theorem \ref{neighboroppositetheorem}
neighboring angles of the triangle have opposite inscribed angle properties. Therefore, we may
select three lines of slope 1 or -1 through the vertices of the triangle in an alternating fashion to construct
three sides of a taxicab circle around the triangle. Since at least one angle of the triangle is completely inscribed,
this can be done in multiple ways. At the completely inscribed angle, a line should be
selected that maximizes the side of the circle enclosed by the 3 lines (the solid lines of Figure \ref{circumcircle}, with the dotted lines showing the alternate choice). By construction,
the fourth side of the taxicab circle will enclose the triangle.
\end{proof}

\begin{figure}[t]
\centerline{\epsffile{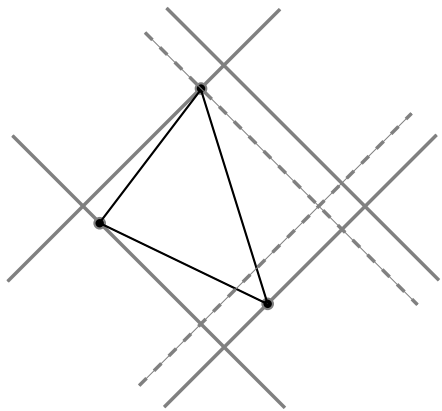}}
\caption{Construction of the taxicab circumcircle of an inscribed triangle}
\label{circumcircle}
\end{figure}

The Triangle Circumcircle Theorem justifies our name for inscribed triangles. This type of triangle is precisely
the kind for which circumscribed circles exist. From this result, we can easily obtain the taxicab version of the
Euclidean geometry ``3 points define a unique circle'' theorem previously investigated from other perspectives in \cite{Colakoglu,Tian}. In particular, the concepts of inscribed angles and triangles are intimately connected with the gradual, steep, and separator lines of the main theorem of \cite{Colakoglu}.

\begin{corollary}
\emph{\textbf{(Three-point Circle Theorem)}} Three non-collinear points lie on a unique
taxicab circle if and only if the triangle they form is inscribed. 
\end{corollary}

As noted by Sowell \cite{Sowell}, taxicab triangle circumcircles are not necessarily unique. All non-unique cases occur when one or more sides of the circumcircle pass through two vertices. This often creates flexibility in defining possible circumcircles. Three cases are shown in Figure \ref{infinitecircumcircles}. In the first example, exactly one side has a slope of 1 or -1 and the other sides are of equal taxicab length and longer than the first side (noting that the triangle does not necessarily need to be acute). In this case a circumcircle can be shifted diagonally within a limited range around the triangle. For the second example, one side has slope 1 or -1 and another side is parallel to an axis. Larger and larger circles encompass this type of triangle. These circles also encompass the third example along with another set of circles in a limited range of movement as in the first example. In this case the triangle is isoceles (but not equilateral) with two sides having slope 1 or -1.

\begin{figure}[t]
\epsfysize=1.8in
\centerline{\epsffile{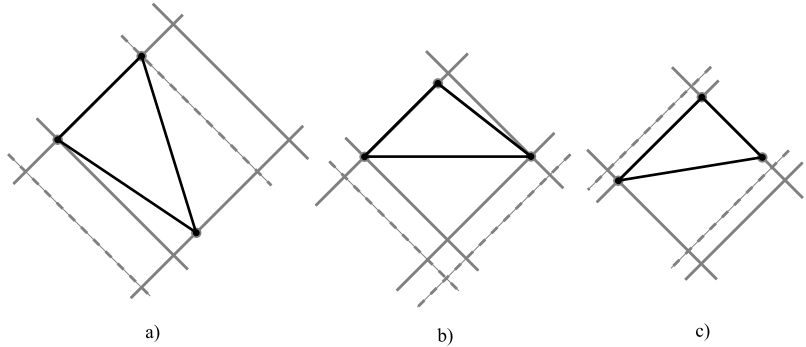}}
\caption{Triangles with infinite circumcircles}
\label{infinitecircumcircles}
\end{figure}

\section{Taxicab Triangle Incircles}

The traditional definition of a triangle incircle is the largest circle
inscribed within the triangle that is tangent to all three sides of
the triangle. Such a circle is guaranteed to exist in Euclidean geometry.
As with circumcircles, a taxicab triangle incircle only exists under
certain conditions. We begin by formalizing the definition of an incircle
in taxicab geometry to clear up concepts like ``tangent'' which
may not transfer precisely.

\begin{definition}
A taxicab triangle {\em incircle}, or inscribed circle, is a taxicab circle entirely contained in a triangle with
three of its corners touching the sides of the triangle.
\end{definition}

\begin{theorem}
\emph{\textbf{(Triangle Incircle Theorem)}} A triangle has a unique taxicab incircle if and only if it is an inscribed triangle.
\end{theorem}
\begin{proof}
Let $\triangle ABC$ be an inscribed triangle. By Theorem \ref{onecompletelyinscribedtheorem}, the triangle has
 at least one completely inscribed angle.
Without loss of generality, let $\gamma$ be an inscribed angle in the triangle at vertex $C$ such that $AC$ and $BC$ have slope
between -1 and 1, inclusive (Figure \ref{incircle}). If the triangle has three inscribed angles,
choose $\gamma$ to be the angle at the intersection of the lines with slope 1 and -1 (see Theorem \ref{threeinscribedanglestheorem}).
For the side $AB$ of the triangle opposite $\gamma$, consider a point $P$ with distance from $B$

\[
r_{\beta}=\frac{\alpha\cdot\overline{AB}}{\alpha+\beta}
\]

The distance from $A$ to $P$ is

\[
r_{\alpha}=\overline{AB}-r_{\beta}=\frac{\beta\cdot\overline{AB}}{\alpha+\beta}
\]

The arcs of taxicab circles of radius $r_{\alpha}$ and $r_{\beta}$
centered on A and B, respectively, intersect $AB$ at the same
point $P$ and by the arc length Theorem 6 of \cite{ThompsonDray} both have a length across the interior of the triangle of

\[
l=r_{\alpha}\cdot\alpha=r_{\beta}\cdot\beta=\frac{\alpha\cdot\beta\cdot\overline{AB}}{\alpha+\beta}
\]

Therefore, these arcs form half a taxicab circle inside the triangle.
The other half of the taxicab circle remains inside the triangle because
the slope of the lines forming the circle are by assumption greater than the slopes
of the lines forming the side of the triangle they intersect.

To prove the converse, assume a triangle has a taxicab incircle. If one or zero sides of the circle overlap the triangle (as in Figure \ref{incircle}),
 both sides of each angle of the triangle are intersected by a line of slope -1 or 1. This implies a
line of the same slope passing through the vertex will remain outside the angle making each angle inscribed.
If two sides of the circle overlap the triangle, the enclosed angle is by definition completely inscribed.
The other two angles are inscribed by the argument above. Therefore, the triangle is an inscribed triangle.
\end{proof}

\begin{figure}[t]
\centerline{\epsffile{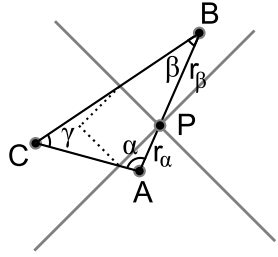}}
\caption{Construction of the taxicab incircle of an inscribed triangle}
\label{incircle}
\end{figure}

For a triangle that is not inscribed, an inscribed taxicab circle still exists in a sense, as noted in passing by Sowell \cite{Sowell}. The enclosed circle technically fails to be an
incircle by our definition because it does not touch some of the sides of the triangle (see Figure \ref{notinscribedtriangle}).
The problem centers on the fact that the triangle is significantly ``wider'' than it is ``tall'' and taxicab
circles are in a sense not as ``flexible'' as Euclidean circles.

\section{Conclusion}

While taxicab incircles and circumcircles exist for some taxicab triangles, their existence is not guaranteed.
The size and position of the angles of the triangle affect the existence of such circles with the
concept of inscribed angles being the defining characteristic. The inscribed angles requirement is perhaps present
in other geometries such as Euclidean geometry, but it may not be as well noted in cases where all angles are inscribed.

\begin{figure}[t]
\centerline{\epsffile{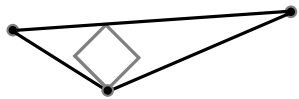}}
\caption{A taxicab circle within a triangle that is not inscribed}
\label{notinscribedtriangle}
\end{figure}

\end{document}